\documentclass[a4paper,10pt]{article}
\usepackage[utf8]{inputenc}

\usepackage{cmap}
\usepackage[T1]{fontenc}

\usepackage{tikz}
\usetikzlibrary{arrows.meta}

\usepackage{subcaption}

\usepackage{tocloft}
\usepackage[english]{babel}
\usepackage{amsfonts,stmaryrd}
\usepackage{amssymb,amsthm}
\usepackage{xcolor,graphicx}
\usepackage{marginnote}
\usepackage{amsmath}
\usepackage{a4wide}
\usepackage[affil-it]{authblk}

\usepackage{epstopdf}
\newcommand{\doi}[1]{\href{http://dx.doi.org/#1}{doi:#1}}

\usepackage{titlesec}
\titlelabel{\thetitle.\quad}

\usepackage[labelsep=period]{caption}

\usepackage{enumitem}

\usepackage{crossreftools}
\makeatletter
\newcommand{\optionaldesc}[2]{%
	\phantomsection
	#1\protected@edef\@currentlabel{#1}\label{#2}%
}
\makeatother

\numberwithin{equation}{section}

\let\OLDthebibliography\thebibliography
\renewcommand\thebibliography[1]{
	\OLDthebibliography{#1}
	\setlength{\parskip}{1pt}
	\setlength{\itemsep}{1pt plus 0.3ex}
}

\usepackage[colorlinks=true,linktocpage,pdfpagelabels,
bookmarksnumbered,bookmarksopen]{hyperref}
\definecolor{ForestGreen}{rgb}{0.1,0.6,0.05}
\definecolor{EgyptBlue}{rgb}{0.063,0.1,0.6}
\hypersetup{
	colorlinks=true,
	linkcolor=EgyptBlue, 
	citecolor=ForestGreen,
	urlcolor=olive
}

\usepackage[hyperpageref]{backref}

\allowdisplaybreaks
\usepackage{mathtools}
\mathtoolsset{showonlyrefs}

\usepackage{orcidlink}

\usepackage{accents}

\def\Wo{W_0^{1,p}(\Omega)}

\newcommand{\cS}{\mathcal{S}}
\newcommand{\cN}{\mathcal N}

\newcommand{\R}{\mathbb R}

\newtheorem{theorem}{Theorem}[section]
\newtheorem{lemma}[theorem]{Lemma}
\newtheorem{proposition}[theorem]{Proposition}
\newtheorem{corollary}[theorem]{Corollary}
\theoremstyle{definition}
\newtheorem{remark}[theorem]{Remark}

\title{\vspace*{-10ex}
On Pleijel-type nodal domain bounds for the $p$-Laplacian
}

\author{Vladimir Bobkov\\}
\date{}

\AtEndDocument{%
	\par
	\medskip
	\bigskip
	\begin{tabular}{@{}l@{}}%
		(V.~Bobkov)\\[0.2em]
		\textsc{Institute of Mathematics, Ufa Federal Research Centre, RAS}\\
		\textsc{Chernyshevsky str. 112, 450008 Ufa, Russia}\\[0.3em]
		\orcidlink{0000-0002-4425-0218} 0000-0002-4425-0218\\[0.3em]
		\textit{E-mail address}: \texttt{bobkov@matem.anrb.ru, bobkovve@gmail.com}\\[1.5em]
\end{tabular}}

\begin{document}
	\maketitle
	\vspace*{-5ex}
	\begin{abstract}
		
		We provide an upper estimate \`a~la Pleijel on the asymptotic number of nodal domains for eigenfunctions corresponding to the cogenus eigenvalues $\{\lambda_k(p;\Omega)\}$ of the $p$-Laplacian in a bounded domain $\Omega$, and identify regimes when the number of nodal domains of the $k$-th eigenfunction is less than $k$ as $k \to +\infty$. 
		As auxiliary results, which also have independent interest, we provide a useful characterization of the cogenus eigenvalues implying their continuity with respect to $p$, justify the Weyl law, and prove the inequality $\lambda_2(p;B) \leq \dots \leq \lambda_{N+1}(p;B) \leq \lambda_\ominus(p)$ in an $N$-dimensional ball $B$, where $\lambda_\ominus(p)$ is an eigenvalue whose eigenfunction 
		has a central section of $B$ as its nodal set.

		
	\end{abstract}

	\begin{quote}	
		\tableofcontents	
		\addtocontents{toc}{\vspace*{-2ex}}
	\end{quote}

\section{Introduction}\label{sec:intro}

Let $\Omega$ be a bounded Lipschitz domain in $\mathbb{R}^N$, $N \geq 2$, and let $p \in (1,+\infty)$. 
Consider the eigenvalue problem
	\begin{equation}\label{eq:D}
		\tag{$\mathcal{D}$}
		\left\{
		\begin{aligned}
			-\Delta_p u &= \lambda |u|^{p-2} u 
			&&\text{in } \Omega, \\
			u &= 0  &&\text{on } \partial \Omega,
		\end{aligned}
		\right.
	\end{equation}
where $\Delta_p u = \text{div}(|\nabla u|^{p-2}\nabla u)$ is the $p$-Laplacian. 
This problem is understood in the weak sense, that is, $\lambda>0$ and $u \in \Wo \setminus \{0\}$ are called an eigenvalue and an eigenfunction of \eqref{eq:D}, provided
\begin{equation}\label{eq:D:weak}
\int_\Omega |\nabla u|^{p-2} \nabla u \nabla \varphi \,dx
= 
\lambda \int_\Omega |u|^{p-2} u \varphi \,dx
\quad \text{for any}~ \varphi \in \Wo.
\end{equation} 
By the classical bootstrap argument, any eigenfunction is bounded, and hence, by the interior regularity results (see, e.g., \cite{tolksdorf}) it belongs to $C^1(\Omega)$. 

In what follows, we frequently use the notation
\begin{equation}\label{eq:S}
\mathcal{S}(\Omega)
=
\left\{u\in W^{1,p}_0(\Omega):~\int_\Omega |u|^p\,dx=1\right\}
\quad \text{and} \quad 
E(u)=\int_\Omega |\nabla u|^p\,dx,
\end{equation}
so that eigenvalues are critical levels of the functional $E \in C^1(\Wo,\mathbb{R})$ over $\mathcal{S}(\Omega)$, and eigenfunctions are corresponding critical points.

\subsection{Eigenvalues and cogenus}\label{subsec:cogenus}
Using the Ljusternik--Schnirelman procedure, one can obtain an infinite sequence of eigenvalues of \eqref{eq:D}. 
One of the classical ways is to introduce the notion of \textit{index}. 
Let $\mathbb{N} =\{1,2,\dots\}$. 
We say that an abstract set $A$ is symmetric whenever $x \in A$ implies $-x \in A$.  
We define an \textit{index} as a set function $i$ satisfying the following assumptions:
\begin{enumerate}[label={$(I_{\arabic*})$}]
    \item\label{I1}
    $i(\emptyset)=0$, and 
    $i(A) \in \mathbb{N}$ for any nonempty, symmetric, compact set $A$ of a topological vector space such that $0 \not\in A$. 
    
    \item\label{I2}
    If $X,Y$ are topological vector spaces, 
    $A \subset X \setminus \{0\}$ is nonempty, symmetric, compact, and $\pi: A \to Y \setminus \{0\}$ is continuous and odd, then
    $$
    i(\pi(A)) \geq i(A).
    $$
\end{enumerate}
Taking $X, Y = W_0^{1,p}(\Omega)$ and $i$ satisfying \ref{I1}, \ref{I2}, we define
\begin{equation}\label{eq:eigen-general} 
	\lambda_k^{(i)}(p;\Omega) 
	= 
	\inf_{A \in \Sigma_k^{(i)}} \max_{u \in A} E(u), \quad k \in \mathbb{N},
\end{equation}
where
\begin{equation}\label{eq:sigmai}
	\Sigma_k^{(i)}
	= 
	\left\{A \subset \mathcal{S}(\Omega):~  A \mbox{ is compact, symmetric, and } i(A)\geq k\right\}.
\end{equation}
For definiteness, we set $\lambda_k^{(i)}(p;\Omega) = +\infty$ if $\Sigma_k^{(i)} = \emptyset$ for some $k \in \mathbb{N}$. 

Since the functional $E$ satisfies the Palais--Smale condition on $\mathcal{S}(\Omega)$ (see, e.g., \cite[Lemma~4]{DR0}), one can use the deformation lemma (see, e.g., \cite[Theorem~4]{DR0}) and the assumptions \ref{I1}, \ref{I2} to obtain the following statement, cf.\ \cite[Theorem~5]{DR0}.
\begin{proposition}\label{prop:eigenvalues}
Let \ref{I1}, \ref{I2} be satisfied. 
If $\Sigma_k^{(i)} \neq \emptyset$ for some $k \in \mathbb{N}$, then $\lambda_k^{(i)}(p;\Omega)$ is an eigenvalue of \eqref{eq:D}. 
\end{proposition}
Although the assumptions \ref{I1}, \ref{I2} are minimal for Proposition~\ref{prop:eigenvalues} to hold, they alone hardly help to produce any reasonable information on the spectrum of the $p$-Laplacian. 
For instance, considering the trivial index $i \equiv 1$, the construction gives only the first eigenvalue, while $\Sigma_k^{(i)}=\emptyset$ for every $k\ge2$.

Among the most common choices of nontrivial indices producing infinite and \textit{divergent} sequences of eigenvalues we recall the Krasnoselskii genus $\gamma_G$ \cite{APA,rabinowitz}, cohomological index  $\gamma_P$ \cite{fadell,perera}, and cogenus $\gamma_F$ \cite{DR0,F}. 
In particular, the Krasnoselskii genus and cogenus are defined, respectively, as
\begin{align}\label{eq:gen}
\gamma_G(A)
&=
\inf
\{
m\in\mathbb{N}:~ \mbox{there exists a continuous odd map } A \to \mathbb{S}^{m-1} 
\},\\
\label{eq:cogen}
\gamma_F(A)
&=
\sup
\{
m\in\mathbb{N}:~ \mbox{there exists a continuous odd map } \mathbb{S}^{m-1} \to A
\}.
\end{align}
Hereinafter, $\mathbb{S}^{k-1}$ is a Euclidean sphere in $\mathbb{R}^k$. 
Note that the definitions \eqref{eq:gen}, \eqref{eq:cogen} work well also for non-compact sets, allowing $\gamma_G(A), \gamma_F(A)=+\infty$.

The eigenvalues constructed via $\gamma_G$, $\gamma_P$, $\gamma_F$ will be denoted as $\lambda_k^G(p;\Omega)$, $\lambda_k^P(p;\Omega)$, $\lambda_k^F(p;\Omega)$, respectively. 
In the same way, the corresponding sets \eqref{eq:sigmai} will be denoted as $\Sigma_k^G$, $\Sigma_k^P$, $\Sigma_k^F$. 
It is known (see \cite[Proposition~4.7]{perera-book}) that 
\begin{equation}\label{eq:ggg0x}
\gamma_F(A) \leq \gamma_P(A) \leq \gamma_G(A)
\quad \text{for any symmetric set}~ A \subset \mathcal{S}(\Omega),
\end{equation}
and hence
\begin{equation}\label{eq:lll0}
\lambda_k^G(p;\Omega) \leq \lambda_k^P(p;\Omega) \leq \lambda_k^F(p;\Omega)
\quad \text{for any}~ k \in \mathbb{N}. 
\end{equation}

In this work, we are mainly interested in the \textit{cogenus eigenvalues}. 
As noted in \cite{perera-book}, they were first defined by Dr\'abek \& Robinson \cite{DR0} in the following, equivalent, way: 
\begin{equation}\label{highereigenvalues2} 
\lambda_k^{DR}(p;\Omega) 
= 
\inf_{A \in \Sigma_k^{DR}} \max_{u \in A} E(u), 
\quad k \in \mathbb{N}, 
\end{equation}
where
\begin{equation}\label{eq:Fk}
\Sigma_k^{DR}
=
\left\{A \subset \mathcal{S}(\Omega):~  A=h(\mathbb{S}^{k-1}) \mbox{ for a continuous odd map } h: \mathbb{S}^{k-1} \to \mathcal{S}(\Omega)\right\}. 
\end{equation}
For clarity, the equivalence is justified in Lemma~\ref{lem:equil} below. 
Notice that, in the linear case $p=2$, the set $\{\lambda_k^{DR}(2;\Omega)\}$ describes the whole spectrum of the Dirichlet Laplacian in $\Omega$, see \cite[Propositions~4.7 and 5.4]{cuesta}. 

For later purposes, we need the continuity of $\lambda_k^{DR}(p;\Omega)$ with respect to $p$.
\begin{proposition}\label{prop:continuity-eigenvalues}
Let $k \in \mathbb{N}$.
Then the map $p \mapsto \lambda_k^{DR}(p;\Omega)$ is continuous in $(1,+\infty)$.
\end{proposition}

A general framework of establishing the continuity of $p \to \lambda_k^{(i)}(p;\Omega)$ is given by Degiovanni \& Marzocchi in \cite{degiovannimarzocchi}, which is based on the $\Gamma$-convergence theory developed in \cite{degiovannimarzocchi2}. 
More precisely, 
since $\Omega$ is a Lipschitz domain,  
\cite[Proposition~2.1 (e), Theorem~4.1 (a), (c), and Corollary~6.2]{degiovannimarzocchi} provide the continuity if the index $i$ satisfies apart from \ref{I1}, \ref{I2} the following two additional assumptions:
\begin{enumerate}[label={$(I_{\arabic*})$},start=3]
    \item\label{I3}
    If $ X$ is a topological vector space and $ A \subset X \setminus \{0\}$ is nonempty, symmetric, and compact, then there exists an open subset $ U$ of $ X \setminus \{0\}$ such that $ A \subset U$ and 
    $$
    i(\widehat{A}) \leq i(A)
    $$
    for any nonempty, symmetric, compact set $ \widehat{A} \subset U$.
    
    \item\label{I4}
    If $ X$ is a normed space with $ 1 \leq \dim X < +\infty$, then
    $$
    i\bigl(\{u \in X : \|u\|_X = 1\}\bigr) = \dim X.
    $$
\end{enumerate}
In fact, according to \cite{degiovannimarzocchi2}, it is sufficient to assume in \ref{I1}, \ref{I2}, \ref{I3} that $X,Y$ are metrizable, locally convex topological vector spaces. 
The Krasnoselskii genus and cohomological index do satisfy \ref{I3}, \ref{I4} (see \cite[Propositions~7.5, 7.7]{rabinowitz} and \cite[Proposition~2.12]{perera-book}, respectively), and hence 
$$
p \mapsto \lambda_k^{G}(p;\Omega)
\quad \text{and} \quad
p \mapsto \lambda_k^{P}(p;\Omega)
\quad \text{are continuous in}~ (1,+\infty).
$$ 
However, the cogenus $\gamma_F$ does not seem to satisfy \ref{I3}, in general. 

In order to overcome this difficulty, we introduce a regularized version of the cogenus, extending the definition of Coffman \cite[Eq.~(3.4)]{coffman-classif} to the general, infinite-dimensional setting. 
For a symmetric set $A \subset X \setminus \{0\}$, denote by $\cN(A)$ the family of open
symmetric neighborhoods of $A$ in $X \setminus \{0\}$. 
Define the \textit{regularized cogenus} as 
\begin{equation}\label{eq:gammac}
\gamma_C(A)=\sup\{m\in \mathbb{N}:\text{ for every }U\in \cN(A)\text{ there exists a
continuous odd map }\mathbb{S}^{m-1}\to U\}. 
\end{equation}
Regarding $\mathbb{S}^{k-1}$ as an equatorial section of $\mathbb{S}^{m-1}$ for some $k \leq m$ (see, e.g., \eqref{eq:i} below), any continuous odd map $\mathbb{S}^{m-1}\to U$ can be restricted to a continuous odd map $\mathbb{S}^{k-1}\to U$. 
Therefore, the definition \eqref{eq:gammac} is equivalent to 
\begin{equation}\label{eq:gammac2}
\gamma_C(A)=\inf_{U\in\cN(A)}\gamma_F(U).
\end{equation}

Let us state two main properties of the regularized cogenus. 
\begin{lemma}\label{lem:indices}
Let $X,Y$ be metrizable, locally convex topological vector spaces. 
Then $\gamma_C$ satisfies the assumptions \ref{I1}--\ref{I4}. 
Moreover, for any nonempty, symmetric, compact set $A \subset X \setminus \{0\}$, we have (cf.\ \eqref{eq:ggg0x})
\begin{equation}\label{eq:FCG0}
\gamma_F(A)\le \gamma_C(A)\le \gamma_G(A).
\end{equation}
\end{lemma}

\begin{lemma}\label{lem:equil}
For any $k \in \mathbb{N}$, we have 
$$
\lambda_k^{DR}(p;\Omega) = \lambda_k^F(p;\Omega) = \lambda_k^C(p;\Omega). 
$$
\end{lemma}

These two lemmas will be proved in Section~\ref{sec:equivalence}. 
In view of them, \cite{degiovannimarzocchi} yields the continuity of $\lambda_k^{DR}(p;\Omega)$ with respect to $p$ stated in Proposition~\ref{prop:continuity-eigenvalues}. 

\medskip
Hereinafter, we mainly work with the cogenus eigenvalues and we denote them simply as $\lambda_k(p;\Omega)$, unless otherwise explicitly stated.

\subsection{Nodal domain bounds}
Let $\varphi_k \in W_0^{1,p}(\Omega)$ be any eigenfunction corresponding to $\lambda_k(p;\Omega)$. 
Recalling that $\varphi_k \in C^1(\Omega) \cap L^\infty(\Omega)$, one can define its nodal set as
$$
\mathcal{Z}(\varphi_k) = \overline{\{x \in \Omega:~ \varphi_k(x) = 0\}}.
$$
Any connected component of $\Omega \setminus \mathcal{Z}(\varphi_k)$ is called a \textit{nodal domain}, and we denote the number of nodal domains of $\varphi_k$ as $\nu(\varphi_k)$.
In the present work, we are interested in estimating this quantity, and we call any such estimate a \textit{nodal domain bound}. 

In the linear case $p=2$, the well-known Courant nodal domain theorem asserts that $\nu(\varphi_k) \leq k$ for all $k \in \mathbb{N}$. 
This estimate can be improved in the regime $k \to +\infty$, as was shown by Pleijel \cite{pleijel} in the planar case and then generalized by B\'erard \& Meyer \cite{BM} to the general case $N \geq 2$, see \cite{Helffer} for an overview.  
Namely, the following asymptotic refinement of the Courant nodal domain bound holds:
\begin{equation}\label{eq:P}
	\limsup_{k \to +\infty} \frac{\nu(\varphi_k)}{k} 
	\leq 
	\frac{(2\pi)^N}{|B_1|^2 \, j_{\frac{N}{2}-1,1}^N}
	< 1,
\end{equation}
where $j_{\mu,1}$ is the first positive zero of the Bessel function $J_{\mu}$, and $B_1$ is an open unit $N$-dimensional ball.

A generalization of the  Courant nodal domain theorem to the nonlinear case $p \neq 2$ was obtained by Dr\'abek \& Robinson \cite{DR} (see also \cite{ananetsouli} for a slightly weaker version). 
It asserts that
\begin{equation}\label{eq:DRcourant}
\nu(\varphi_k) \leq 2k-2
\quad \text{for all }
k \geq 2.
\end{equation}
In addition to \eqref{eq:DRcourant}, \cite[Theorems~3.1 and 3.2]{DR} show that the classical Courant nodal domain bound 
\begin{equation}\label{eq:courant}
\nu(\varphi_k) \leq k 
\quad \text{for a given } k \in \mathbb{N}
\end{equation}
holds if either $\varphi_k$ satisfies the unique continuation property (which is unknown even for the second eigenfunction when $p \neq 2$) or $\lambda_k < \lambda_{k+1}$.  
Moreover, by \cite[Theorem~3.4]{DR}, if $\nu(\varphi_k) = k +l$ for some $\varphi_k$ and $l \in \{1,\dots,k-2\}$, then there exists a $k$-th eigenfunction $\psi_k$ such that $\nu(\psi_k) = k-l$. 
In particular, \eqref{eq:courant} always holds for \textit{some} eigenfunction $\varphi_k$. 

The unconditional nodal domain bound \eqref{eq:DRcourant} implies 
\begin{equation}\label{eq:P1}
\mathfrak{P}(p;\Omega) 
:= 
\limsup_{k \to +\infty} \frac{\nu(\varphi_k)}{k} \leq 2.
\end{equation}
The aim of the present work is to investigate the extent to which the core arguments of Pleijel \cite{pleijel} (see also \cite{BM,Helffer}) 
can be generalized to the nonlinear setting and lead to upper bounds on the constant $\mathfrak{P}(p;\Omega)$ better than \eqref{eq:P1}.

\subsection{Main results}

We start with a general scheme, which is due to \cite{pleijel}, in essence; see also \cite{BM,Helffer}. 
Considering the restriction $\varphi_k|_{\Omega_i}$ of $\varphi_k$ to some nodal domain $\Omega_i$ and recalling that $\varphi_k \in C(\Omega)$, we have $\varphi_k|_{\Omega_i} \in W_0^{1,p}(\Omega_i) \setminus \{0\}$, see \cite[Lemma~5.6]{CuestaFucik}. 
Hence, by the weak formulation \eqref{eq:D:weak}, $\varphi_k|_{\Omega_i}$ is a sign-constant eigenfunction of the $p$-Laplacian in $\Omega_i$ corresponding to the eigenvalue $\lambda_k(p;\Omega)$. 
Using the variational characterization of $\lambda_1(p;\Omega_i)$, we get
\begin{equation}\label{eq:lambdan=lambda0}
\lambda_1(p;\Omega_i) 
\leq
\frac{\int_{\Omega_i} |\nabla \varphi_k|^p \,dx}{\int_{\Omega_i} |\varphi_k|^p \,dx}
=
\lambda_k(p;\Omega), \quad i \in \{1,\dots,\nu(\varphi_k)\}. 
\end{equation}
Since the only sign-constant eigenfunction is the first one (see, e.g., \cite{kawohl-lind}), we actually have 
\begin{equation}\label{eq:lambdan=lambda1}
\lambda_1(p;\Omega_i) = \lambda_k(p;\Omega), \quad i \in \{1,\dots,\nu(\varphi_k)\}. 
\end{equation}
Consider now the scaling-free quantity $|\Omega_i| \lambda_1^{N/p}(p;\Omega_i)$ and let $\mathcal{L}_i > 0$ be any of its lower bounds, i.e.,
\begin{equation}\label{eq:lambdai}
|\Omega_i| \lambda_1^{N/p}(p;\Omega_i) \geq \mathcal{L}_i, \quad i \in \{1,\dots,\nu(\varphi_k)\}. 
\end{equation}
Noting that $|\Omega_1| + \dots + |\Omega_{\nu(\varphi_k)}| \leq |\Omega|$ and using \eqref{eq:lambdan=lambda0} (or \eqref{eq:lambdan=lambda1}), we sum the inequalities \eqref{eq:lambdai} over $i$ and arrive at
\begin{equation}\label{eq:plej1}
|\Omega| \lambda_k^{N/p}(p;\Omega) 
\geq
\sum_{i=1}^{\nu(\varphi_k)} \mathcal{L}_i 
\geq
\nu(\varphi_k) \min\{\mathcal{L}_1,\dots,\mathcal{L}_{\nu(\varphi_k)}\}. 
\end{equation}
From here, the number $\nu(\varphi_k)$ of nodal domains of $\varphi_k$ gets the following estimate, which is valid for any $k \in \mathbb{N}$, cf.\ \cite[Corollary~4.1]{cuesta} and \cite[Lemma~2.2]{DR}.

\begin{proposition}\label{prop:pleij1}
Let $k \in \mathbb{N}$. 
Let $\mathcal{L}_i>0$ be given by \eqref{eq:lambdai}, $i \in \{1,\dots,\nu(\varphi_k)\}$. 
Then 
\begin{equation}\label{eq:mun1}
\nu(\varphi_k) 
\leq 
|\Omega| \lambda_k^{N/p}(p;\Omega)
\left(
\min\{\mathcal{L}_1,\dots,\mathcal{L}_{\nu(\varphi_k)}\}
\right)^{-1}.
\end{equation}
\end{proposition}

The simplest and arguably most useful choice for $\mathcal{L}_i$ comes from the classical 
Faber--Krahn inequality. 
It holds for any open set of finite measure, giving, in particular, the uniform bound
\begin{equation}\label{eq:FK}
|\Omega_i| \lambda_1^{N/p}(p;\Omega_i) 
\geq 
\mathcal{L}_i := |B_1| \lambda_1^{N/p}(p;B_1), \quad i \in \{1,\dots,\nu(\varphi_k)\},
\end{equation}
where we recall that $B_1$ is an open unit ball in $\mathbb{R}^N$, for definiteness. 

\begin{remark}
In some cases, the Faber--Krahn inequality \eqref{eq:FK} can be substituted by a better lower bound. 
In the linear case $p=2$, this route leads to fine estimates and even exact values of $\mathfrak{P}(2;\Omega)$ when $\Omega$ is a domain with separable geometry, 
see, e.g., \cite{Bob2,Helffer} and references therein. 	
However, in the nonlinear case $p \neq 2$, geometric properties of $\{\Omega_i\}$ are not available even for the simplest domains, to the best of our knowledge. 
\end{remark}

In the asymptotic regime $k \to +\infty$, $\lambda_k(p;\Omega)$ obeys the Weyl law, cf.\ \cite{F,mazur}. 	
\begin{theorem}\label{thm:weyl}
There exists $C_{\mathcal{W}} = C_{\mathcal{W}}(p,N) > 0$ such that
\begin{equation}\label{eq:weyl}
\lim_{k\to +\infty} \frac{\lambda_k(p;\Omega)}{k^{p/N}} = \frac{C_{\mathcal{W}}}{|\Omega|^{p/N}}.
\end{equation}
\end{theorem}

In the linear case $p=2$, it is well known that
\begin{equation}\label{eq:weylp=2}
C_{\mathcal{W}} = \frac{(2 \pi)^{2}}{|B_1|^{2/N}}.
\end{equation}
One way to derive this value is to employ the exact structure of the eigenspaces of the Laplacian in an $N$-dimensional cube.  
However, no such information is known in the case $p \neq 2$, and we can only provide the following semi-explicit upper bound on $C_{\mathcal{W}}$, see Section~\ref{sec:Cl} for definitions. 
\begin{proposition}\label{prop:weyl}
Let $m \in \mathbb{N}$. 
Let $\Lambda$ be a full-rank point lattice in $\mathbb{R}^N$, and let $\Sigma$ be a corresponding fundamental domain (tiling domain), which is a Lipschitz domain.  
Then
\begin{equation}\label{eq:prop:weyl}
C_{\mathcal{W}}
\leq 
|\Sigma|^{p/N} \, \frac{\lambda_m(p;\Sigma)}{m^{p/N}}. 
\end{equation}
\end{proposition}

This proposition can be seen as a nonlinear counterpart of \cite{polya} on the validity of P\'olya's conjecture for tiling domains.

\medskip
Combining now the estimate \eqref{eq:P1}, Proposition~\ref{prop:pleij1}, the lower bound \eqref{eq:FK}, Theorem~\ref{thm:weyl}, and Proposition~\ref{prop:weyl},  
we arrive at the following general statement.

\begin{theorem}\label{theorem:pleijel-general}
Let $m \in \mathbb{N}$. 
Let $\Lambda$ be a full-rank point lattice in $\mathbb{R}^N$, and let $\Sigma$ be a corresponding fundamental domain, which is a Lipschitz domain.  
Then
\begin{equation}\label{eq:upperb-general}
\mathfrak{P}(p;\Omega) 
\leq 
\min\left\{
2,
\frac{1}{|B_1|} \frac{C_{\mathcal{W}}^{N/p}}{\lambda_1^{N/p}(p;B_1)}
\right\}
\leq 
\min\left\{
2,
\frac{|\Sigma|}{m |B_1|} \frac{\lambda_m^{N/p}(p;\Sigma)}{\lambda_1^{N/p}(p;B_1)}
\right\}. 
\end{equation}
\end{theorem}

\begin{corollary}\label{cor:m=1}
Let $m=1$. 
Let $\Lambda$ and $\Sigma$ in Theorem~\ref{theorem:pleijel-general} be such that $\Sigma$ has inradius $1$. 
Then $\lambda_1(p;\Sigma) \leq \lambda_1(p;B_1)$, and hence
\begin{equation}\label{eq:upperb-general2}
\mathfrak{P}(p;\Omega) 
\leq 
\min\left\{
2,
\frac{|\Sigma|}{|B_1|}
\right\}.
\end{equation}
Noting that $\delta(\Lambda) := |B_1|/|\Sigma|$ is the sphere packing density of $\Lambda$, we use its best values in low dimensions (see, e.g., \cite[Chapter~I, Table~1.2]{conwaysloane}) to get from \eqref{eq:upperb-general2} that
\begin{equation}\label{eq:upperb-general4}
\mathfrak{P}(p;\Omega) 
\leq 
\begin{cases}
\frac{2\sqrt{3}}{\pi} = 1.10265... &\text{for}~ N=2 ~\text{(hexagonal lattice)},\\
\frac{3\sqrt{2}}{\pi} = 1.35047... &\text{for}~ N=3 ~\text{(face-centered cubic lattice)},\\
\frac{16}{\pi^2} = 1.62113... &\text{for}~ N=4 ~\text{(checkerboard lattice $D_4$)}.
\end{cases}
\end{equation}
\end{corollary}

The available results on the sphere packing density of lattices in higher dimensions $N \geq 5$ (see, e.g., \cite[Chapter~I, Table~1.2]{conwaysloane}) give $\delta^{-1}(\Lambda) > 2$, and hence \eqref{eq:upperb-general2} reduces to the estimate \eqref{eq:P1}. 
Since $\lambda_1^{1/p}(p;\Sigma)\to 1$ and $\lambda_1^{1/p}(p;B_1) \to 1$ as $p \to +\infty$ by \cite{JLM}, the upper bound \eqref{eq:upperb-general} with $m=1$ has no significant advantage over the upper bound \eqref{eq:upperb-general2} for sufficiently large $p$. 
Moreover, we did not find a reasonable improvement of \eqref{eq:upperb-general2} for sufficiently large $p$ by taking $m \in \{2,\dots,N+1\}$ in \eqref{eq:upperb-general}. 

Nonetheless, we are able to identify two regimes when $\mathfrak{P}(p;\Omega) < 1$. 
\begin{proposition}\label{prop:continuity}
For every $\varepsilon>0$, there exists $\sigma>0$ such that 
$$
\mathfrak{P}(p;\Omega)
\leq 
\min\left\{
2,
\frac{(2\pi)^N}{|B_1|^2 j_{\frac{N}{2}-1,1}^N} + \varepsilon
\right\}
\quad \text{for any}~ p \in (2-\sigma,2+\sigma). 
$$
In particular, there exist $p_* \in [1,2)$ and $p^*>2$ such that 
$$
\mathfrak{P}(p;\Omega) < 1 
\quad \text{for any}~ p \in (p_*,p^*). 
$$
\end{proposition}

\begin{proposition}\label{prop:weyl2}
Let $m=N+1$ in \eqref{eq:upperb-general}. 
Then 
\begin{equation}\label{eq:upperb-general6}
	\limsup_{p \to 1} \mathfrak{P}(p;\Omega) 
	\leq 
	\begin{cases}
	\frac{2\sqrt{3}}{3\pi} \left(\frac{3.15429...}{2}\right)^2 
		=
		0.91424... &\text{for}~ N=2,\\
	\frac{3\sqrt{2}}{4\pi} \left(\frac{4.2644}{3}\right)^3 = 0.96969... &\text{for}~ N=3.
	\end{cases}
\end{equation}
\end{proposition}

We see that, for the $p$-Laplacian, the quantitative side of the Pleijel analysis significantly relies on the value of the Weyl constant $C_{\mathcal{W}}$. 
Without having either an \textit{explicit} value of $C_{\mathcal{W}}$, or at least an \textit{explicit} value of the estimate \eqref{eq:prop:weyl} for sufficiently large $m$, we are not able to derive even the Courant nodal domain bound $\nu(\varphi_k) \leq k$ for all $p \neq 2$ as $k \to +\infty$. 
The exact value of $C_{\mathcal{W}}$ remains a deep problem related to the structure of the spectrum of the $p$-Laplacian in the $N$-dimensional cube. 
We refer to \cite[Section~6]{F} for a conjectured value.

\medskip
The rest of this work is structured as follows.
In Section~\ref{sec:equivalence}, we prove the properties of the regularized cogenus stated in 
Lemmas~\ref{lem:indices} and \ref{lem:equil}. 
In Section~\ref{sec:weyl}, we prove Theorem~\ref{thm:weyl} and Proposition~\ref{prop:weyl}. 
Section~\ref{sec:proof} is devoted to the proofs of the remaining assertions -- Propositions~\ref{prop:continuity} and \ref{prop:weyl2}.
Counterparts of the main results for other types of minimax eigenvalues are discussed in Section~\ref{sec:krasnosel}.

\section{Properties of the regularized cogenus}\label{sec:equivalence}

In this section, we prove Lemmas~\ref{lem:indices} and \ref{lem:equil}.

\begin{proof}[Proof of Lemma~\ref{lem:indices}]
In what follows, $A$ is a nonempty, symmetric, compact set such that $0 \not\in A$. 
We start by proving the inequalities
\begin{equation}\label{eq:FCG1}
\gamma_F(A)\le \gamma_C(A)\le \gamma_G(A).
\end{equation}
If, for some $m \in \mathbb{N}$, there exists an odd map $f: \mathbb{S}^{m-1}\to A$, then $f$ can be seen as $f:\mathbb{S}^{m-1} \to U$ for any $U \in \mathcal{N}(A)$. 
Taking the supremum among such $m$, we arrive at $\gamma_F(A)\le \gamma_C(A)$. 

Since $A$ is compact, we have $n=\gamma_G(A)<+\infty$, see, e.g., \cite[Proposition~7.5]{rabinowitz}. 
Let $f:A\to\mathbb{S}^{n-1}$ be continuous and odd, and we can regard $f$ as $f:A \to \mathbb{R}^n$. 
By the Tietze extension theorem (see, e.g., \cite{dugunji}), $f$ extends to a continuous map $F:X\to\R^n$. Define
$F_{\mathrm{\mathrm{odd}}}(u)=(F(u)-F(-u))/2$, so that $F_{\mathrm{odd}}$ is odd and $F_{\mathrm{odd}}=f$ on $A$. 
Since
$|F_{\mathrm{odd}}|=1$ on $A$, the set $V=\{u\in X:~F_{\mathrm{odd}}(u)\ne 0\}$ is an open symmetric neighborhood of $A$, i.e., $V \in \mathcal{N}(A)$. 
The map $u\mapsto F_{\mathrm{odd}}(u)/|F_{\mathrm{odd}}(u)|$ is a continuous odd map from $V$ to $\mathbb{S}^{n-1}$. 
Therefore, by the Borsuk--Ulam theorem, there is no continuous odd map $\mathbb{S}^m\to V$ for $m \geq  n$, which yields $\gamma_F(V)\le n$. 
Recalling that $V \in \mathcal{N}(A)$, we get from the definition \eqref{eq:gammac2} of $\gamma_C(A)$ that 
$$
\gamma_C(A) \leq \gamma_F(V) \leq n = \gamma_G(A).
$$
This finishes the proof of \eqref{eq:FCG1}. 

Let us now justify that $\gamma_C$ satisfies the assumptions \ref{I1}--\ref{I4}. 

\ref{I1} Take any $u\in A$ and define a map $f: \mathbb{S}^0=\{-1,1\}\to A$ as $f(1)=u$ and $f(-1)=-u$. 
Clearly, $f$ is continuous and odd, and hence $\gamma_F(A)\ge1$.
In view of \eqref{eq:FCG1}, $\gamma_C(A) \geq 1$ and it is finite.

\ref{I2} Let $m = \gamma_C(A)$, and let $V$ be an open symmetric neighborhood of $\pi(A)$ in $Y\setminus\{0\}$. 
Similarly to the first part of the proof, the Dugundji extension theorem \cite{dugunji} guarantees that $\pi$ can be extended to a continuous odd map $P_{\mathrm{odd}}:X\to Y$. 
(Here we used the assumption that $X$ is metrizable and $Y$ is a locally convex topological vector space.) 
Since $P_{\mathrm{odd}}(A)\subset V$ and $V$ is open, there exists an open symmetric neighborhood $W$ of $A$ in $X\setminus\{0\}$ such that $P_{\mathrm{odd}}(W)\subset V$. 
By $m = \gamma_C(A)$, there is a continuous odd map
$h:\mathbb{S}^{m-1}\to W$. Then $P_{\mathrm{odd}}\circ h:\mathbb{S}^{m-1}\to V$ is continuous and odd. 
Since $V$ was arbitrary, we get $\gamma_C(\pi(A))\ge m$. 

\ref{I3} Let $m=\gamma_C(A)$. 
By the definition \eqref{eq:gammac}, there exists an open symmetric neighborhood $U$ of $A$ in
$X\setminus\{0\}$ such that there is no continuous odd map $\mathbb{S}^m\to U$.
If $\widehat A\subset U$ is nonempty, symmetric, and compact, then $U$ is
also an admissible neighborhood of $\widehat A$. Hence, 
$\gamma_C(\widehat A)\le m=\gamma_C(A)$.

\ref{I4} 
If $\dim X=n$, then for $S_X := \{u \in X : \|u\|_X = 1\}$ we have $\gamma_F(S_X) = n$ (by the Borsuk--Ulam theorem) and $\gamma_G(S_X) = n$ (see, e.g., \cite[Proposition~7.7]{rabinowitz}), so that $\gamma_C(S_X)=n$ follows from \eqref{eq:FCG1}.
\end{proof}

\begin{proof}[Proof of Lemma~\ref{lem:equil}]
First, we prove that $\lambda_k^F(p;\Omega)= \lambda_k^{DR}(p;\Omega)$.
Let us show that $\lambda_k^F(p;\Omega)\leq\lambda_k^{DR}(p;\Omega)$. 
Take any $A\in\Sigma_k^{DR}$. By the definition \eqref{highereigenvalues2}, there exists a continuous odd map
$h:\mathbb{S}^{k-1}\to\mathcal{S}(\Omega)$ such that $A=h(\mathbb{S}^{k-1})$. Since
$\mathbb{S}^{k-1}$ is compact, the set $A$ is compact. 
Moreover, the oddness of $h$ implies that $A$ is symmetric. 
Regarding $h$ as a map from $\mathbb{S}^{k-1}$ onto its image $A$, the definition \eqref{eq:cogen} of the cogenus gives $\gamma_F(A)\geq k$. 
Thus, we have $A\in\Sigma_k^F$, and hence
$$
    \lambda_k^F(p;\Omega)
    \leq
    \max_{u\in A}E(u).
$$
Taking the infimum over $A\in\Sigma_k^{DR}$, we obtain
$\lambda_k^F(p;\Omega)\leq\lambda_k^{DR}(p;\Omega)$.

We now prove the reverse inequality. Let $A\in\Sigma_k^F$. Since $m:=\gamma_F(A)\geq k$, there exists a continuous odd map $g:\mathbb{S}^{m-1}\to A$. 
Consider a map
\begin{equation}\label{eq:i}
    \iota:\mathbb{S}^{k-1} \to \mathbb{S}^{m-1},
    \qquad
    \iota(x_1,\ldots,x_k)
    =
    (x_1,\ldots,x_k,0,\ldots,0).
\end{equation}
Since $\iota$ is continuous and odd, so is the map $h:=g\circ\iota:\mathbb{S}^{k-1}\to A \subset \mathcal{S}(\Omega)$. Therefore, the set $B:=h(\mathbb{S}^{k-1})$ belongs to
$\Sigma_k^{DR}$. Since $B\subset A$, it follows that
$$
    \lambda_k^{DR}(p;\Omega)
    \leq
    \max_{u\in B}E(u)
    \leq
    \max_{u\in A}E(u).
$$
Taking the infimum over $A\in\Sigma_k^F$, we conclude that
$\lambda_k^{DR}(p;\Omega)\leq\lambda_k^F(p;\Omega)$. 
Therefore, we deduce that $\lambda_k^F(p;\Omega) = \lambda_k^{DR}(p;\Omega)$.

Second, we prove that $\lambda_k^C(p;\Omega)=\lambda_k^{F}(p;\Omega)$. 
The inequality $\lambda_k^C(p;\Omega) \leq \lambda_k^{F}(p;\Omega)$ follows from \eqref{eq:FCG0}. 
To show the reverse inequality, for any $\varepsilon>0$ let $A\subset\cS(\Omega)$ be a compact symmetric set such that  $m := \gamma_C(A)\ge k$ and $\max_{u \in A} E(u) \leq \lambda_k^C(p;\Omega)+\varepsilon$. 
The set 
$$
U_\varepsilon=\left\{u\in W_0^{1,p}(\Omega) \setminus \{0\}:~ 
\frac{E(u)}{\int_\Omega |u|^p\,dx}<\lambda_k^C(p;\Omega)+2\varepsilon
\right\}
$$ 
is an open symmetric neighborhood of $A$. 
By the definition of $\gamma_C(A)$, there exists a continuous odd map $h:\mathbb{S}^{m-1}\to U_\varepsilon$. 
Normalizing $h$ as $\widehat{h}(z) = h(z)/(\int_\Omega |h(z)|^p \,dx)^{1/p}$, we get $\widehat{h}(\mathbb{S}^{m-1}) \in \Sigma_m^{DR}$. 
Therefore, we obtain
$$
\lambda_k^{F}(p;\Omega)
=
\lambda_k^{DR}(p;\Omega)
\leq
\lambda_m^{DR}(p;\Omega)
\leq
\max_{u \in \widehat{h}(\mathbb{S}^{m-1})} E(u) 
=
\max_{u \in {h}(\mathbb{S}^{m-1})} \frac{E(u)}{\int_\Omega |u|^p\,dx} 
\leq 
\lambda_k^C(p;\Omega)+2\varepsilon.
$$ 
Letting $\varepsilon \to 0$, we arrive at the inequality
$\lambda_k^{F}(p;\Omega)\le\lambda_k^C(p;\Omega)$, which finishes the proof. 
\end{proof}

Recall that, in view of Lemmas~\ref{lem:indices} and \ref{lem:equil}, the results of \cite{degiovannimarzocchi} imply the validity of Proposition~\ref{prop:continuity-eigenvalues} on the continuity of the cogenus eigenvalues with respect to $p$.

\section{Weyl law}\label{sec:weyl}

The aim of this section is to prove Theorem~\ref{thm:weyl} and Proposition~\ref{prop:weyl}. 
Let us introduce some useful notation. 
Following \cite{F,mazur}, for $\lambda>0$, let 
$$
M_\lambda^0(\Omega)
=
\left\{
u \in \mathcal{S}(\Omega):~ E(u) < \lambda 
\right\},
$$
and denote by $N_\Omega^0(\lambda)$ the counting function for the cogenus eigenvalues $\{\lambda_k(p;\Omega)\}$, that is,
$$
N_\Omega^0(\lambda) =
\# \{k \in \mathbb{N}:~ \lambda_k(p;\Omega) < \lambda\}.
$$
Let us provide three main properties of $N_\Omega^0(\lambda)$ needed to establish Theorem~\ref{thm:weyl}. 

\begin{lemma}\label{lem:equality}{(Characterization)}
Let $\lambda>0$. Then 
$$
N_\Omega^0(\lambda) = \gamma_F(M_\lambda^0(\Omega)).
$$
\end{lemma}
\begin{proof}
We first prove that, for every $k\in\mathbb N$,
\begin{equation}\label{eq:equivalence:subadd}
\lambda_k(p;\Omega)<\lambda
\quad\Longleftrightarrow\quad
\gamma_F(M_\lambda^0(\Omega))\ge k.
\end{equation}
Assume that $\lambda_k(p;\Omega)<\lambda$. 
By the definition \eqref{eq:eigen-general} (with $i=\gamma_F$) of $\lambda_k(p;\Omega)$,
there exists a compact symmetric set $A\subset\mathcal{S}(\Omega)$ such that
$\gamma_F(A)\ge k$ and $\max_{u\in A}E(u)<\lambda$. 
Hence, $A\subset M_\lambda^0(\Omega)$. 
It is clear from the definition \eqref{eq:cogen} that the cogenus is monotone with
respect to set inclusion, which yields 
$$
\gamma_F(M_\lambda^0(\Omega))
\ge \gamma_F(A)
\ge k.
$$
Conversely, assume that
$\gamma_F(M_\lambda^0(\Omega))\ge k$. By the definition of cogenus,
there exist $m \geq  k$ and a continuous odd map $g:\mathbb{S}^{m-1}\to M_\lambda^0(\Omega)$. 
As in the proof of Lemma~\ref{lem:equil}, let $\iota:\mathbb{S}^{k-1}\to\mathbb{S}^{m-1}$ be defined by \eqref{eq:i}. 
Then $f:=g\circ\iota$ is a continuous odd map from $\mathbb{S}^{k-1}$ to $M_\lambda^0(\Omega)$.
Denote $A=f(\mathbb{S}^{k-1})$. 
Then $A\subset M_\lambda^0(\Omega)\subset\mathcal{S}(\Omega)$, and $A$ is compact
and symmetric. Moreover, since $f:\mathbb{S}^{k-1}\to A$ is continuous and odd,
we have $\gamma_F(A)\ge k$. 
That is, $A$ is admissible for the definition  of $\lambda_k(p;\Omega)$, and we have
$$
\lambda_k(p;\Omega) \leq \max_{u\in A}E(u)<\lambda.
$$
This proves the equivalence \eqref{eq:equivalence:subadd}. 

Denoting $m=\gamma_F(M_\lambda^0(\Omega)) \in \mathbb{N} \cup \{+\infty\}$, \eqref{eq:equivalence:subadd} gives
$$
\{k\in\mathbb N:~\lambda_k(p;\Omega)<\lambda\}
=
\{k\in\mathbb N:~k \leq  m\}.
$$
If $m<+\infty$, then the latter set is $\{1,\dots,m\}$, and hence it has
cardinality $m$. If $m=+\infty$, then both sets are infinite. 
Consequently, we get the desired equality:
\begin{equation*}
N_\Omega^0(\lambda) \equiv \#\{k\in\mathbb N:\lambda_k(p;\Omega)<\lambda\}
=
\gamma_F(M_\lambda^0(\Omega)) = m.
\qedhere
\end{equation*}
\end{proof}

\begin{lemma}\label{lem:superadditivity}{(Superadditivity)}
Let $\lambda>0$. 
Let $V$ and $W$ be two disjoint subdomains of $\Omega$. 
Then
\begin{equation}\label{eq:lem:superadd}
N_\Omega^0(\lambda)
\geq
N_V^0(\lambda)
+
N_W^0(\lambda).
\end{equation}
\end{lemma}
\begin{proof}
The result is essentially due to \cite[Lemma~3]{F}, which states that 
$$
\gamma_F(M_\lambda^0(\Omega))
\geq
\gamma_F(M_\lambda^0(V))
+
\gamma_F(M_\lambda^0(W)). 
$$
Thanks to Lemma~\ref{lem:equality}, this inequality is equivalent to \eqref{eq:lem:superadd}. 
\end{proof}

\begin{lemma}\label{lem:scaling}{(Scaling)}
Let $\lambda>0$ and $a>0$.
Then 
\begin{equation}\label{eq:scaling}
N_{a\Omega}^0(\lambda)
=
N_{\Omega}^0(a^p\lambda),
\end{equation}
where $a\Omega = \{ax \in \mathbb{R}^N:~ x \in \Omega\}$.
\end{lemma}
\begin{proof}
The proof is evident. For instance, it can be established verbatim as \cite[Proposition~5.3]{mazur}, by noting that $\gamma_F$ is stable under odd homeomorphisms, see the assumption \ref{I2} in Section~\ref{subsec:cogenus}.
\end{proof}

Now we are ready to prove the Weyl law for the cogenus eigenvalues $\{\lambda_k(p;\Omega)\}$. 
\begin{proof}[Proof of Theorem~\ref{thm:weyl}]
The proof is largely based on the results of \cite{F,mazur}, and we provide only missing details.  
Let $K$ be an open unit $N$-dimensional cube. 
Arguing verbatim as in \cite[Lemma~6.1]{mazur} with the help of Lemmas~\ref{lem:superadditivity} and \ref{lem:scaling}, we conclude that the function $\lambda \mapsto \lambda^{-N/p} N_K^0(\lambda)$ tends to $c_0 \in [0,+\infty]$ as $\lambda \to +\infty$. 
By \cite[Remark, p.~1066]{F}, which asserts the existence of $C_i=C_i(p,N)$, $i=1,2$, such that 
\begin{equation}\label{eq:fried10}
C_1 |K| \lambda^{N/p}
\leq
\gamma_F(M_\lambda^0(K))
\leq
C_2 |K| \lambda^{N/p}
\quad \text{as}~ \lambda \to +\infty, 
\end{equation}
and in view of Lemma~\ref{lem:equality}, we obtain $c_0 \in (0,+\infty)$. 
Finally, \cite[Theorem~6.3]{mazur} extends in a standard way the Weyl law from the cube $K$ to a general bounded Lipschitz domain $\Omega$. 
This result does not rely on particular properties of the cohomological index, and hence it remains valid also for the cogenus. 
This completes the proof. 
\end{proof}

\begin{remark}\label{rem:weyl}
In much the same way as in the proof of Theorem~\ref{thm:weyl}, the arguments of \cite{F,mazur} would also translate to the case of the Krasnoselskii genus $\gamma_G$, provided one knows that the corresponding counting function satisfies the superadditivity relation of Lemma~\ref{lem:superadditivity}.  
However, unlike the case of the cohomological index $\gamma_P$ and cogenus $\gamma_F$, this property does not seem to be known for $\gamma_G$. 
This fact highlights that, despite the result of \cite{mazur} on the Weyl law for the cohomological index eigenvalues, the original conjecture of Friedlander \cite{F} on the validity of the Weyl law for the Krasnoselskii genus eigenvalues \textit{remains an open problem}. 
\end{remark}

\subsection{Estimate on the Weyl constant}\label{sec:Cl}
The aim of this subsection is to establish the upper bound on $C_{\mathcal{W}}$ stated in Proposition~\ref{prop:weyl}. 
We start with several definitions (see, e.g., \cite{conwaysloane}) and auxiliary results. 
Let $\Lambda\subset \mathbb{R}^N$ be a full-rank lattice, that is, $\Lambda$ is the image of $\mathbb{Z}^N$ under an invertible linear transformation of $\mathbb{R}^N$. 
Let $\Sigma\subset \mathbb{R}^N$ be a fundamental domain for $\Lambda$, that is,
$$
\mathbb{R}^N
=
\bigcup_{z\in\Lambda} \overline{z+\Sigma},
\quad \text{where}~ 
(z+\Sigma)\cap(z'+\Sigma)=\emptyset
\quad\text{for}~ z\neq z'.
$$
Assume that $\Sigma$ is sufficiently regular, e.g., a Lipschitz domain. 
We also assume, for definiteness, that $\Sigma$ has inradius $1$, so that there exists $a\in\Sigma$ such that $B_1(a)\subset \Sigma$. 
Then the family of balls $\{B_1(a+z):~z\in\Lambda\}$ 
is a lattice sphere packing, and its packing density can be defined as 
$$
\delta(\Lambda)
=
\frac{|B_1|}{|\Sigma|}. 
$$
For $h>0$, we denote $\Sigma_h =h\Sigma$, so that $\Sigma_1=\Sigma$, $\Sigma_h$ has inradius $h$, and
$$
|\Sigma_h|=h^N|\Sigma|,
\qquad
\lambda_m(p;\Sigma_h)=h^{-p}\lambda_m(p;\Sigma)
\quad \text{for any}~ m \in \mathbb{N}. 
$$
Define 
\begin{equation}\label{eq:M}
M_\Omega(h)
=
\sup_{t\in \Sigma_h}
\#
\left\{
z\in\Lambda:~
t+h(z+\Sigma)\subset\Omega
\right\}.
\end{equation}
\begin{lemma}\label{lem:Mh-asymp}
We have 
\begin{equation}\label{eq:Mh-asymp}
M_\Omega(h)
=
\frac{|\Omega|}{h^N|\Sigma|}(1+o(1))
\quad\text{as}~ h \to 0.
\end{equation}
\end{lemma}
\begin{proof}
Since all shifted fundamental domains in the definition \eqref{eq:M} are disjoint subsets of $\Omega$, we have 
\begin{equation}\label{eq:proof:wer1}
M_\Omega(h)h^N|\Sigma|\leq |\Omega|,
\end{equation}
which gives the upper bound in \eqref{eq:Mh-asymp}.  
Let us find the lower bound. 
Denote
$$
\Omega_\rho
=
\{x\in\Omega:~ \text{dist}(x,\partial \Omega)>\rho\}.
$$
Since $\Omega$ is open and bounded, we have $|\Omega_\rho|\to|\Omega|$ as $\rho \to 0$. 
Let us fix $R>0$ such that $\Sigma\subset B_R(0)$. 
If $t+hz\in\Omega_{hR}$, then
$$
t+h(z+\Sigma)\subset B_{hR}(t+hz)\subset\Omega.
$$
Therefore, 
\begin{equation}\label{eq:momegalower1}
M_\Omega(h)
\geq
\sup_{t\in \Sigma_h}
\#\{z\in\Lambda:~ t+hz\in\Omega_{hR}\}.
\end{equation}
Integrating over $t\in \Sigma_h$, we get
\begin{equation}\label{eq:proof:second}
\int_{\Sigma_h}
\#\{z\in\Lambda:~ t+hz\in\Omega_{hR}\}\,dt
=
\sum_{z\in\Lambda}
\int_{\Sigma_h}\chi_{\Omega_{hR}}(t+hz)\,dt       
		=
\int_{\mathbb{R}^N}
\chi_{\Omega_{hR}}(x)\,dx
=
|\Omega_{hR}|,
\end{equation}
where $\chi_{\Omega_{hR}}$ is the characteristic function of $\Omega_{hR}$. 
(In the second equality in \eqref{eq:proof:second}, we used the fact that $|\partial \Sigma_h| = 0$ by the Lipschitzness of $\Sigma$.)
Thus, there exists $t_h\in \Sigma_h$ such that
$$
\#\{z\in\Lambda:~ t_h+hz\in\Omega_{hR}\}
\geq
\frac{|\Omega_{hR}|}{|\Sigma_h|}
=
\frac{|\Omega_{hR}|}{h^N|\Sigma|},
$$
and hence \eqref{eq:momegalower1} yields 
$$
M_\Omega(h)
\geq
\frac{|\Omega_{hR}|}{h^N|\Sigma|}.
$$
Recalling the upper estimate \eqref{eq:proof:wer1}, we let $h \to 0$ and deduce the desired asymptotic \eqref{eq:Mh-asymp}. 
\end{proof}

Let us fix some $m\in\mathbb N$. 
For any $k\in\mathbb N$, put
$$
n_k=\left\lceil \frac{k}{m}\right\rceil,
$$
so that $n_k m \geq k$, and define
\begin{equation}\label{eq:hkm}
h_{k,m}
=
\sup\{h>0:~ M_\Omega(h)\geq n_k\}.
\end{equation}

\begin{lemma}\label{lem:hk}
Let $m \in \mathbb{N}$. Then
\begin{equation}\label{eq:lem:hk}
h_{k,m}
=
\left(
\frac{m |\Omega|}{k|\Sigma|}
\right)^{1/N}
(1+o(1))
\quad\text{as}~k\to +\infty.
\end{equation}
\end{lemma}
\begin{proof}
Since $n_k=\lceil k/m\rceil$ and $m$ is fixed, we have
\begin{equation}\label{eq:proof:lem:hk}
n_k=\frac{k}{m}(1+o(1))
\quad\text{as }k\to +\infty,
\end{equation}
and hence it is sufficient to prove that 
\begin{equation}\label{eq:hkm-nk}
h_{k,m}
=
\left(
\frac{|\Omega|}{n_k|\Sigma|}
\right)^{1/N}
(1+o(1))
\quad\text{as }k\to +\infty.
\end{equation}
Fix some $\varepsilon \in (0,1)$. 
Then, by Lemma~\ref{lem:Mh-asymp}, there exists $h^* = h^*(\varepsilon)>0$ such that for any $h \in (0,h^*)$ we have 
\begin{equation}\label{eq:lowerupper1}
(1-\varepsilon) \frac{|\Omega|}{h^{N} |\Sigma|} 
\leq
M_\Omega(h)
\leq
(1+\varepsilon) \frac{|\Omega|}{h^{N}|\Sigma|}.
\end{equation}
Define
$$
a_k^- =
\left(
\frac{(1-\varepsilon)|\Omega|}{n_k |\Sigma|}
\right)^{1/N}
\quad \text{and} \quad 
a_k^+ =
\left(
\frac{(1+2\varepsilon)|\Omega|}{n_k |\Sigma|}
\right)^{1/N}.
$$
Clearly, $a_k^\pm \in (0,h^*)$ for any sufficiently large $k$, so that \eqref{eq:lowerupper1} holds for $h = a_k^\pm$. 
On one hand, 
$$
M_\Omega(a_k^-)
\geq
(1-\varepsilon) \frac{|\Omega|}{(a_k^-)^{N} |\Sigma|}
=
n_k,
$$
and hence \eqref{eq:hkm} yields $h_{k,m}\geq a_k^-$. 
On the other hand, for any $h \in [a_k^+, h^*)$ the upper estimate in \eqref{eq:lowerupper1} gives
$$
M_\Omega(h)
\leq
(1+\varepsilon)\frac{|\Omega|}{h^{N} |\Sigma|}
\leq
(1+\varepsilon)\frac{|\Omega|}{(a_k^+)^{N} |\Sigma|}
=
\frac{1+\varepsilon}{1+2\varepsilon} n_k
<
n_k,
$$
while for any $h \geq h^*$ the definition of $M_\Omega(h)$ yields (cf.\ \eqref{eq:proof:wer1})
$$
M_\Omega(h)
\leq
\frac{|\Omega|}{h^{N} |\Sigma|}
\leq
\frac{|\Omega|}{(h^*)^{N} |\Sigma|} < n_k
\quad \text{for any sufficiently large}~ k. 
$$
Thus, there is no $h \geq a_k^+$ admissible for the definition \eqref{eq:hkm} of $h_{k,m}$, implying $h_{k,m} \leq a_k^+$. 
Consequently, for a given $\varepsilon>0$ and any sufficiently large $k$, we get
$$
\left(
\frac{(1-\varepsilon)|\Omega|}{n_k |\Sigma|}
\right)^{1/N}
\leq
h_{k,m}
\leq 
\left(
\frac{(1+2\varepsilon)|\Omega|}{n_k |\Sigma|}
\right)^{1/N}.
$$
Letting $k\to +\infty$ and then $\varepsilon\to0$, we obtain
\eqref{eq:hkm-nk}, which is equivalent to
\eqref{eq:lem:hk} in view of \eqref{eq:proof:lem:hk}. 
\end{proof}

\begin{lemma}\label{lem:lambda-upperbound}
Let $k,m \in \mathbb{N}$. 
Then, for any $\eta \in (0,1)$, 
\begin{equation}\label{eq:lambda-upperbound}
\lambda_k(p;\Omega)
\leq 
((1-\eta)h_{k,m})^{-p}\lambda_m(p;\Sigma).
\end{equation}
\end{lemma}
\begin{proof}
It follows from the definition \eqref{eq:hkm} of $h_{k,m}$ that for any $\eta \in (0,1)$ 
there exists
$h>0$ such that
$$
(1-\eta)h_{k,m} \leq h\leq h_{k,m}
\quad \text{and} \quad 
M_\Omega(h)\geq n_k.
$$
Hence, for some $t\in \Sigma_h$, there are distinct points
$z_1,\dots,z_{n_k}\in\Lambda$ such that the domains
$$
Q_i=t + h (z_i+\Sigma),
\quad i=1,\dots,n_k,
$$
are mutually disjoint and contained in $\Omega$.  
By Lemma~\ref{lem:superadditivity}, for any $\lambda>\lambda_m(p;Q_i)$ we have
$N_\Omega^0(\lambda)\ge n_k m$, and hence $\lambda_{n_km}(p;\Omega)<\lambda$.
Sending $\lambda \to \lambda_m(p;Q_i)$ and recalling that $n_k m \geq k$, we arrive at
\begin{equation*}
\lambda_k(p;\Omega) 
\leq 
\lambda_m(p;Q_i)
=
h^{-p}\lambda_m(p;\Sigma)
\leq 
((1-\eta)h_{k,m})^{-p}\lambda_m(p;\Sigma).
\qedhere
\end{equation*}
\end{proof}

Finally, using Lemmas~\ref{lem:hk} and \ref{lem:lambda-upperbound}, we are able to justify Proposition~\ref{prop:weyl}.

\begin{proof}[Proof of Proposition~\ref{prop:weyl}]
For any $m \in \mathbb{N}$ and $\eta \in (0,1)$, Lemmas~\ref{lem:hk} and \ref{lem:lambda-upperbound}  give
$$
\lambda_k(p;\Omega)
\leq 
(1-\eta)^{-p}h_{k,m}^{-p}\lambda_m(p;\Sigma)
= 
 (1-\eta)^{-p}
 \lambda_m(p;\Sigma)
 \left(
 \frac{k |\Sigma|}{m |\Omega|}
 \right)^{p/N}
 (1+o(1))
 \quad \text{as}~ k \to +\infty.   
$$
Dividing by $k^{p/N}$, taking the limit as $k\to +\infty$ (this limit exists by Theorem~\ref{thm:weyl}), and then letting $\eta \to 0$, we arrive at
\begin{equation}\label{eq:prop:Cl}
\lim_{k\to +\infty}
\frac{\lambda_k(p;\Omega)}{k^{p/N}}
\leq
\lambda_m(p;\Sigma)
\left(
\frac{|\Sigma|}{m |\Omega|}
\right)^{p/N}.
\end{equation}
This is equivalent to the desired estimate \eqref{eq:prop:weyl} on $C_{\mathcal{W}}$. 
\end{proof}

\section{Proof of Propositions~\ref{prop:continuity} and \ref{prop:weyl2}}\label{sec:proof}
Throughout this section, we use the expanded notation $C_{\mathcal{W}}(p)$ to emphasize the dependence on $p$.

\begin{proof}[Proof of Proposition~\ref{prop:continuity}] 
We employ the continuity arguments. 
Recall that $\{\lambda_k(2;\Omega)\}$ exhausts the whole spectrum of the Dirichlet Laplacian in $\Omega$, see \cite[Propositions~4.7 and 5.4]{cuesta}. 
Let us take any fundamental domain $\Sigma$ and assume that it is sufficiently regular, e.g., a Lipschitz domain. 
Let us fix any $\varepsilon>0$. 
Then, by Theorem~\ref{thm:weyl} (see also \cite{polya}), there exists $M=M(\varepsilon)>0$ such that 
\begin{equation}\label{eq:proof:prop1:1}
\frac{|\Sigma| \lambda_m^{N/2}(2;\Sigma)}{m}
\leq 
C_{\mathcal{W}}^{N/2}(2) + \varepsilon
\quad \text{for any}~ m \geq M. 
\end{equation}
Let us fix any $m \geq M$.
By Proposition~\ref{prop:continuity-eigenvalues}, the map $p \mapsto \lambda_m(p;\Sigma)$ is continuous, and hence so is $p \mapsto \lambda_m^{N/p}(p;\Sigma)$. 
In particular, for the given $\varepsilon>0$ and $m \geq M$ there exists $\sigma = \sigma(\varepsilon,m)$ such that
\begin{equation}\label{eq:proof:prop1:2}
\frac{|\Sigma|\lambda_m^{N/p}(p;\Sigma)}{m}
\leq
\frac{|\Sigma|\lambda_m^{N/2}(2;\Sigma)}{m}
+\varepsilon
\quad \text{and} \quad 
|B_1|\lambda_1^{N/p}(p;B_1)
\geq
|B_1|\lambda_1^{N/2}(2;B_1)
-\varepsilon 
\end{equation}
for any $p \in (2-\sigma,2+\sigma)$. 

Substituting now \eqref{eq:proof:prop1:1} and \eqref{eq:proof:prop1:2} into \eqref{eq:upperb-general}, we arrive at 
$$
\mathfrak{P}(p;\Omega)
\leq
\frac{|\Sigma|}{m |B_1|} \frac{\lambda_m^{N/p}(p;\Sigma)}{\lambda_1^{N/p}(p;B_1)}
\leq
\frac{C_{\mathcal{W}}^{N/2}(2) + 2\varepsilon}{|B_1|\lambda_1^{N/2}(2;B_1)
-\varepsilon}
\leq
\frac{C_{\mathcal{W}}^{N/2}(2)}{|B_1|\lambda_1^{N/2}(2;B_1)}
+
\widetilde{\varepsilon}
$$
for any $p \in (2-\sigma,2+\sigma)$, where $\widetilde{\varepsilon} = O(\varepsilon)$ as $\varepsilon \to 0$. 
Indeed, assuming $\varepsilon \leq |B_1|\lambda_1^{N/2}(2;B_1)/2$, we get
$$
\mathfrak{P}(p;\Omega)
\leq
\frac{C_{\mathcal{W}}^{N/2}(2)}{|B_1|\lambda_1^{N/2}(2;B_1)}
+
\frac{C_{\mathcal{W}}^{N/2}(2) + 2|B_1|\lambda_1^{N/2}(2;B_1)}{|B_1|^2\lambda_1^{N}(2;B_1)}
\, 2\varepsilon.
$$

Observe that, for $p=2$, we have by \eqref{eq:weylp=2}
\begin{equation}\label{eq:proof:prop1:3}
\frac{C_{\mathcal{W}}^{N/2}(2)}{|B_1|\lambda_1^{N/2}(2;B_1)}
=
\frac{(2\pi)^N}{|B_1|^2 j_{\frac{N}{2}-1,1}^N},
\end{equation}
where the right-hand side is exactly the original Pleijel constant \eqref{eq:P}. 
Moreover, the constant on the right-hand side of \eqref{eq:proof:prop1:3} is strictly less than $1$, see \cite{BM}. 
Thus, additionally using the upper bound $\mathfrak{P}(p;\Omega) \leq 2$ (see \eqref{eq:P1}), the claim of the proposition follows. 
\end{proof}

\medskip
Prior to the proof of Proposition~\ref{prop:weyl2}, we introduce a few notions and establish two auxiliary results. 
For any $x \in \mathbb{S}^{N-1}$, define a unit half-ball
	$$
	B_1^x = \{z \in B_1:~ \langle z, x \rangle > 0\}.
	$$
Let $v_x \in W_0^{1,p}(B_1^x)$ be the first eigenfunction of the $p$-Laplacian in $B_1^x$ such that $v_x > 0$ and $\int_{B_1^x} |v_x|^p \,dz = 1$. We extend it by zero outside of $B_1^x$ so that $v_x \in W_0^{1,p}(B_1)$.
By the simplicity of the first eigenvalue, for any $x,y \in \mathbb{S}^{N-1}$ there exists $R \in SO(N)$ such that $v_x \circ R=v_y$. 
It can be observed (e.g., by \cite[Lemma~2.8]{ABP}) that $v_x - v_{-x}$ is an eigenfunction of the $p$-Laplacian in $B_1$ corresponding to the eigenvalue
$$
\lambda_\ominus(p) := \lambda_1(p;B_1^x). 
$$

The first auxiliary statement is a cogenus counterpart of an ``almost-multiplicity'' result \cite[Theorem~1.1, Eq.~(1.10)]{ABP} obtained for the Krasnoselskii genus eigenvalues.  
\begin{lemma}\label{lem:multiplicity}
We have 
$$
\lambda_2(p;B_1) \leq \dots \leq \lambda_{N+1}(p;B_1) \leq \lambda_\ominus(p). 
$$
\end{lemma}
\begin{proof}

Define
$$
    \alpha(s)
    =
    \frac{1}{\bigl(1+|2s-1|^p\bigr)^{1/p}}
    \quad \text{and} \quad 
    \beta(s)
    =
    \frac{2s-1}{\bigl(1+|2s-1|^p\bigr)^{1/p}}
    \quad \text{for}~ s \in [0,1],
$$
and set 
$$
    H(x,s)=\alpha(s)v_x+\beta(s)v_{-x}.
$$
We have
\begin{equation}\label{eq:hhh}
    H(x,0)=\frac{v_x - v_{-x}}{2^{1/p}}
    \quad \text{and} \quad 
    H(x,1)=\frac{v_x + v_{-x}}{2^{1/p}} \geq 0 
    \quad \text{in}~ B_1.
\end{equation}
In particular, $H(\cdot,0)$ is odd. 
Moreover, observe that $|\alpha(s)|^p+|\beta(s)|^p=1$, and hence, by disjointness of the
supports of $v_x$ and $v_{-x}$, we obtain
$$
    H(x,s)\in\mathcal{S}(B_1)
    \quad \text{and} \quad 
    E(H(x,s))=\lambda_\ominus(p)
    \quad \text{for any}~ x \in \mathbb{S}^{N-1}, ~s \in [0,1].
$$
It is not hard to see that $H: \mathbb{S}^{N-1} \times [0,1] \to \mathcal{S}(B_1)$ is continuous. 

Let now $\varphi_1 \in \mathcal{S}(B_1)$ be the positive first eigenfunction of the $p$-Laplacian in $B_1$. 
Recalling from \eqref{eq:hhh} that $H(\cdot,1) \geq 0$, for any $x \in \mathbb{S}^{N-1}$ consider 
$$
    G(x,s)
    =
    \bigl((1-s)H(x,1)^p+s\varphi_1^p\bigr)^{1/p},
    \quad s\in[0,1].
$$
Clearly, we have 
$G(x,s)\in\mathcal{S}(B_1)$ and $G: \mathbb{S}^{N-1} \times [0,1] \to \mathcal{S}(B_1)$ is continuous. 
The hidden convexity inequality (see, e.g., \cite[Lemma~1]{diazsaa}) gives
\begin{align}
    E(G(x,s))
    &\leq
    (1-s)E(H(x,1))+sE(\varphi_1)\\
    &=
    (1-s) \lambda_\ominus(p) + s \lambda_1(p;B_1) 
    \leq \lambda_\ominus(p) \quad \text{for any}~ x \in \mathbb{S}^{N-1},~ s\in[0,1].
\end{align}

Consider a map $F: \mathbb{S}^{N-1} \times [0,1] \to \mathcal{S}(B_1)$ defined as $F(\cdot,s) = H(\cdot,2s)$ for $s \in [0,1/2]$, and $F(\cdot,s) = G(\cdot,2s-1)$ for $s \in [1/2,1]$, so that $F$ is continuous and 
\begin{equation}\label{eq:lll1}
    F(x,0)=H(x,0),\quad
    F(x,1)=\varphi_1,\quad
    E(F(x,s))\leq\lambda_\ominus(p)
    \quad 
    \text{for any}~ x \in \mathbb{S}^{N-1},~ s\in[0,1].
\end{equation}
Our aim is to extend $F$ in a continuous odd way to $\mathbb{S}^{N}$. 
For this purpose, we write
$$
    \mathbb{S}^N
    =
    \{(y,t)\in\mathbb R^N\times\mathbb R:\ |y|^2+t^2=1\},
$$
and define an extension map $\Phi:\mathbb{S}^N\to\mathcal{S}(B_1)$ as follows. If
$(y,t)\in\mathbb{S}^N$ and $t\geq0$, set
$$
    \Phi(y,t)
     =
    \begin{cases}
F\left(\frac{y}{|y|},t\right) &\text{for}~ y\neq0,\\[3mm]
\varphi_1 &\text{for}~ y=0 ~(\text{equivalently, } t=1).
    \end{cases}
$$
If $t\leq0$, set
$$
    \Phi(y,t)
     =
    \begin{cases}
-F\left(-\frac{y}{|y|},-t\right) &\text{for}~ y\neq0,\\[3mm]
-\varphi_1 &\text{for}~ y=0  ~(\text{equivalently, } t=-1).
    \end{cases}
$$
This definition is consistent on the equator $t=0$, since then
$|y|=1$ and
$$
   F(-x,0)=H(-x,0) = -H(x,0) = -F(x,0),
$$
and we see that $\Phi$ is odd.
Let us show that $\Phi$ is continuous. 
This is immediate for $t \in (-1,1)$.  
If $t=1$, then, recalling that 
$F(x,1)=\varphi_1$ for every $x\in\mathbb{S}^{N-1}$, the continuity of $\Phi$ follows
from the uniform continuity of $F$ on the compact set
$\mathbb{S}^{N-1}\times[0,1]$. The argument for $t=-1$ is the same.

Since $\Phi:\mathbb{S}^N \to \mathcal{S}(B_1)$ is a
continuous odd map, we get $\Phi(\mathbb{S}^N) \in \Sigma_{N+1}^{DR}$, see \eqref{eq:Fk}.
Recalling \eqref{eq:lll1}, we obtain from the definition \eqref{highereigenvalues2} of $\lambda_{N+1}(p;B_1)$ that 
$$
\lambda_{N+1}(p;B_1) \leq \max_{u \in \Phi(\mathbb{S}^N)} E(u) \leq \lambda_\ominus(p). 
$$
This completes the proof. 
\end{proof}

It is known that 
\begin{equation}\label{eq:cheeger}
\lambda_1(p;\Omega) \to h(\Omega) := \inf_{E \subset \Omega} \frac{P(E)}{|E|}
\quad \text{as}~ p \to 1,
\end{equation}
where $h(\Omega)$ is the Cheeger constant of $\Omega$, see, e.g.,  \cite{Leonardi2015,Parini1} for an overview. 
Here, $P(E)$ denotes the distributional perimeter of $E$ in $\R^N$, which coincides with the $(N-1)$-dimensional Hausdorff measure of $\partial E$ if $E$ is sufficiently smooth, e.g., a Lipschitz domain.  
Using the isoperimetric inequality, it is not hard to observe that
\begin{equation}\label{eq:hb1}
h(B_1)=N \quad \text{for any}~ N \geq 2.
\end{equation}
Moreover, it is known from \cite[Eq.~(4.5)]{Parini1} that 
\begin{equation}\label{eq:hb1+2d}
h(B_1^+) = 3.15429... \quad \text{for}~ N=2,
\end{equation}
where $B_1^+ := B_1 \cap \{x_N>0\}$ is a unit half-ball, $x=(x_1,\dots,x_N)$.  
We would like to estimate $h(B_1^+)$ for $N =3$. 

\begin{lemma}\label{lem:cheeger}
Let $N=3$. Then $h(B_1^+) < 4.2644$. 
\end{lemma}
\begin{proof} 
Let us fix $\rho\in(0,1/2)$ and consider the inner parallel body of $B_1^+$:
$$
    K_\rho
     =
    \left\{
x\in\overline{B_1^+}:~
\text{dist}(x,\partial B_1^+)\geq\rho
    \right\}.
$$
Since $\partial B_1^+$ consists of the unit hemisphere and the
flat disk contained in $\{x_3=0\}$, we have
$$
    K_\rho
    =
    \left\{
x\in\mathbb{R}^3:~
|x|\leq 1-\rho,~ x_3\geq\rho
    \right\},
$$
see Figure~\ref{fig:1}. 
Thus, $K_\rho$ is a solid spherical cap of the ball $B_{1-\rho}$.
The flat portion of $\partial K_\rho$ is a disk of radius $a=\sqrt{1-2\rho}$.

\begin{figure}[!ht]
\centering
\begin{tikzpicture}[scale=4, line cap=round, line join=round, >=Stealth]

%------------------------------------------------
% Parameters for the picture
%------------------------------------------------
\def\rhx{0.28}       % any number in (0,1/2) is fine
\pgfmathsetmacro{\Rin}{1-\rhx}       % radius of the inner spherical cap
\pgfmathsetmacro{\a}{sqrt(1-2*\rhx)} % a = sqrt((1-rho)^2-rho^2)
\pgfmathsetmacro{\zetx}{\rhx/\Rin}   % zeta = rho/(1-rho)
\pgfmathsetmacro{\alpha}{asin(\zetx)}% in degrees
\pgfmathsetmacro{\xzet}{sqrt(1-\zetx*\zetx)}

%------------------------------------------------
% Outer half-disk B_1^+
%------------------------------------------------
\draw[thick] (-1,0) arc[start angle=180,end angle=0,radius=1] -- (1,0);
\draw[thick] (-1,0) -- (1,0);

%-----------------------------------------------
% Competitor E_rho
%------------------------------------------------
\fill[gray!10]
  (-\a,0)
  arc[start angle=270,end angle=180-\alpha,radius=\rhx]
  arc[start angle=180-\alpha,end angle=\alpha,radius=1]
  arc[start angle=\alpha,end angle=-90,radius=\rhx]
  -- cycle;

\draw[very thick, black!75!black]
  (-\a,0)
  arc[start angle=270,end angle=180-\alpha,radius=\rhx]
  arc[start angle=180-\alpha,end angle=\alpha,radius=1]
  arc[start angle=\alpha,end angle=-90,radius=\rhx]
  -- cycle;
  
  \draw[dotted, black!75!black]
    (-\a+\rhx,\rhx)
    arc[start angle=0,end angle=360,radius=\rhx]
    -- cycle;

\draw[<-, solid, black!75!black]
  (\a+0.07, \a*\rhx+0.07*\rhx+0.095)
  arc[start angle=\alpha,end angle=-90,radius=0.07];
%  -- cycle;       

%------------------------------------------------
% Inner parallel set K_rho
%------------------------------------------------
\draw[dashed, black!70!black, thick]
  (-\a,\rhx) -- (\a,\rhx);

\draw[dashed, black!70!black, thick]
  ({\Rin*cos(180-\alpha)},{\Rin*sin(180-\alpha)})
  arc[start angle=180-\alpha,end angle=\alpha,radius=\Rin];
 
%------------------------------------------------
% Symmetry axis
%------------------------------------------------
\draw[->] (0,0) -- (0,1.1);
\node at (-0.07,1.08) {$x_3$};
\draw[->] (-1.1,0) -- (1.1,0);
\node at (1.1,-0.05) {$x_1,x_2$};

%------------------------------------------------
% Construction lines for a and zeta
%------------------------------------------------
\draw[densely dotted] (0,\zetx) --  (\xzet,\zetx);
\draw[densely dotted] (\a,0) -- node [left] {$\rho$} (\a,\rhx);

\draw[dotted] (0,0) -- (\a+\rhx,\zetx);

% label a
\draw[thin] (0,0) -- node [below] {$a$} (\a,0);

%------------------------------------------------
% Points used in the construction
%------------------------------------------------
\fill (\a,\rhx) circle (0.012);
\fill (-\a,\rhx) circle (0.012);
\fill (\xzet,\zetx) circle (0.012);
\fill (0,\zetx) circle (0.012);
\fill (-\xzet,\zetx) circle (0.012);

%------------------------------------------------
% Labels
%------------------------------------------------
\node at (-0.75,0.82) {$B_1^+$};
\node at (-0.05,\zetx) {$\zeta$};
\node at (-0.25,0.47) {$K_{\rho}$};
\node at (0.55,0.65) {$E_{\rho}$};

\node at (-0.25,0.07) {$D$};
\node at (0.5,0.97) {$C$};
\node at (-0.82,0.15) {$T$};

\node at (0.76,0.2) {$\sigma$};

\end{tikzpicture}
\caption{A schematic plot for the proof of Lemma~\ref{lem:cheeger}. The gray set is $E_\rho$.}
\label{fig:1}
\end{figure}
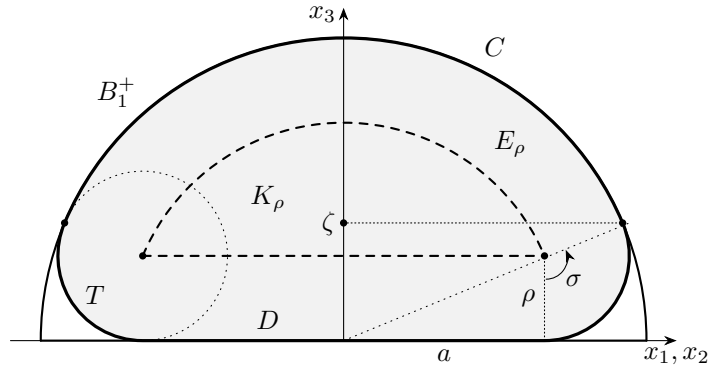

Let us consider a subdomain $E_\rho$ of $B_1^+$ defined as
$E_\rho =\text{Int}\bigl(K_\rho+\overline{B_\rho}\bigr)$, where ``$+$'' is the Minkowski sum. 
In particular, we have $h(B_1^+)\leq P(E_\rho)/|E_\rho|$. 
We provide closed form expressions for $P(E_\rho)$ and $|E_\rho|$. 
For convenience, introduce the notation
$$
    \zeta=\frac{\rho}{1-\rho}
    \quad \text{and} \quad 
    \sigma=\frac{\pi}{2}+\arcsin\zeta.
$$
The boundary $\partial E_\rho$ consists of three pieces $D$, $C$, $T$, where $D$ is a disk of radius $a$, $C$ is a cap of the unit sphere, and $T$ is a part of a torus joining them.
It is clear that $|D| = \pi a^2$. 
The cap $C$ lies between the heights $x_3=\zeta$ and $x_3=1$, so $|C| = 2\pi(1-\zeta)$. 
Let us find $|T|$. We parameterize $T$ by
$$
    X(\varphi,\theta)
    =
    \bigl((a+\rho\sin\theta)\cos\varphi,\,
  (a+\rho\sin\theta)\sin\varphi,\,
  \rho-\rho\cos\theta\bigr),
$$
where $0\leq\varphi<2\pi$ and $0\leq\theta\leq\sigma$. 
The corresponding Jacobian determinant is $\rho(a+\rho\sin\theta)$, so we have
$$
|T|=
    2\pi\rho
    \int_0^\sigma (a+\rho\sin\theta)\,d\theta
    =
    2\pi\rho
    \bigl(a\sigma+\rho(1+\zeta)\bigr).
$$
Therefore, 
\begin{align}
    P(E_\rho) 
    &= |D|+|C|+|T|\\
    \label{eq:Prho}
    &=
    \pi (1-2\rho)+2\pi \frac{1-2\rho}{1-\rho}
    +2\pi\rho
    \biggl[
\sqrt{1-2\rho} \left(
    \frac{\pi}{2}+\arcsin \frac{\rho}{1-\rho}
\right)
+\frac{\rho}{1-\rho}
    \biggr].
\end{align}
Let us now calculate $|E_\rho|$. 
For $0\leq x_3\leq\zeta$, the horizontal section of
$E_\rho$ is a disk of radius
$$
    a+\sqrt{\rho^2-(x_3-\rho)^2}
    =
    a+\sqrt{2\rho x_3-x_3^2}.
$$
For $\zeta\leq x_3\leq1$, the boundary of $E_\rho$ agrees with $\partial B_1$, and hence 
the corresponding section has radius $\sqrt{1-x_3^2}$. 
Thus, we get
\begin{align}
    |E_\rho|
    &=
    \pi\int_0^\zeta
    \left(
a+\sqrt{2\rho z-z^2}
    \right)^2\,dz 
    +\pi\int_\zeta^1(1-z^2)\,dz\\
    \label{eq:Erho}
    &=\pi\left[
    \frac{2}{3}
    -2\rho^2
    +\rho^2\sqrt{1-2\rho}
    \left(
    \frac{\pi}{2}
    +\arcsin\frac{\rho}{1-\rho}
    \right)
    \right].
\end{align}
Substituting, for instance, $\rho=0.2956$, 
into \eqref{eq:Prho} and \eqref{eq:Erho}, we obtain
$$
    P(E_\rho)= 8.08976...
    \quad \text{and} \quad 
    |E_\rho|= 1.89708...
$$
Consequently, $h(B_1^+)\leq P(E_{0.2956})/|E_{0.2956}| < 4.2644$. 
\end{proof}

\begin{remark}
Since $B_1^+$ is a domain of revolution around the $x_N$-axis, one can use the method of \cite{BobParRot} and derive a semi-explicit expression for $h(B_1^+)$ for any $N \geq 3$. 
In particular, the free boundary of the minimizing set for $h(B_1^+)$ is made of a part of a \textit{nodoid} rather than a torus.
A corresponding numerical investigation in the case $N=3$ indicates that $h(B_1^+) = 4.26395...$, which is rather close to the upper bound $4.2644$ from Lemma~\ref{lem:cheeger}. 
We decided to avoid presenting the expression for $h(B_1^+)$ following \cite{BobParRot}, since the upper bound $4.2644$ is simpler to derive and its value is sufficient for our purposes. 
\end{remark}

Using Lemmas~\ref{lem:multiplicity} and \ref{lem:cheeger}, we justify Proposition~\ref{prop:weyl2}.
\begin{proof}[Proof of Proposition~\ref{prop:weyl2}]
Assume that the fundamental domain $\Sigma$ has inradius $1$. 
By the domain monotonicity, we have $\lambda_m(p;\Sigma) \leq \lambda_m(p;B_1)$ for any $m \in \mathbb{N}$. 
Hence, taking $m=N+1$ in Theorem~\ref{theorem:pleijel-general} and using Lemma~\ref{lem:multiplicity}, we get
\begin{equation}\label{eq:proof:prop2:1}
\mathfrak{P}(p;\Omega) 
\leq 
\frac{|\Sigma|}{(N+1) |B_1|} \frac{\lambda_{N+1}^{N/p}(p;B_1)}{\lambda_1^{N/p}(p;B_1)}
\leq
\frac{|\Sigma|}{(N+1) |B_1|} \frac{\lambda_\ominus^{N/p}(p)}{\lambda_1^{N/p}(p;B_1)}. 
\end{equation}
Therefore, using \eqref{eq:cheeger} and the definition of $\lambda_\ominus(p)$, we arrive at
$$
\limsup_{p \to 1} \mathfrak{P}(p;\Omega) 
\leq
\frac{|\Sigma|}{(N+1) |B_1|} \frac{h^{N}(B_1^+)}{h^{N}(B_1)}. 
$$
Let $N=2$. We take $\Sigma$ to be a regular hexagon of inradius $1$, so that $|\Sigma|/|B_1| = 2\sqrt{3}/\pi$. 
Using \eqref{eq:hb1} and \eqref{eq:hb1+2d}, we arrive at
$$
\limsup_{p \to 1} \mathfrak{P}(p;\Omega) \leq
\frac{2\sqrt{3}}{3\pi} \left(\frac{3.15429...}{2}\right)^2 
		=
		0.91424...
$$
Let $N=3$. 
We take $\Lambda$ to be the face-centered cubic lattice (see, e.g., \cite{conwaysloane}), so that $|\Sigma|/|B_1| = 3\sqrt{2}/\pi$. 
(For instance, one can let $\Sigma$ be a rhombic dodecahedron.) 
Using \eqref{eq:hb1} and Lemma~\ref{lem:cheeger}, we arrive at
$$
\limsup_{p \to 1} \mathfrak{P}(p;\Omega) \leq
\frac{3\sqrt{2}}{4\pi} \left(\frac{4.2644}{3}\right)^3 = 0.96969...
$$
The proof is complete.
\end{proof}

\section{Counterparts for other indices}\label{sec:krasnosel}
	
The Courant-type nodal domain bounds \eqref{eq:DRcourant} and \eqref{eq:courant} are established in \cite[Section~3]{DR} for the cogenus eigenfunctions. 
The arguments of \cite{DR} hold verbatim for eigenfunctions corresponding to $\lambda_k^{(i)}(p;\Omega)$, provided the index $i$ satisfies the assumptions \ref{I1}, \ref{I2}, \ref{I4} of Section~\ref{subsec:cogenus}, since only compact symmetric subsets of $\mathcal{S}(\Omega)$ homeomorphic to $\mathbb{S}^{m-1}$ (for certain $m \in \mathbb{N}$) are involved in the arguments. 
In particular, the eigenfunctions corresponding to the Krasnoselskii genus eigenvalues $\lambda_k^G(p;\Omega)$ and the cohomological index eigenvalues $\lambda_k^P(p;\Omega)$ obey the nodal domain bound $\nu(\varphi_k) \leq 2k-2$ for any $k \geq 2$. 

Assume, in addition to \ref{I1}, \ref{I2}, \ref{I4}, that the index $i$ is such that the counting function for the sequence $\{\lambda_k^{(i)}(p;\Omega)\}$ satisfies the superadditivity relation as in Lemma~\ref{lem:superadditivity}.
Then the following upper bound can be established by the same arguments as in Section~\ref{sec:Cl}:
\begin{equation}\label{eq:prop:Clgen}
\limsup_{k\to +\infty}
\frac{\lambda_k^{(i)}(p;\Omega)}{k^{p/N}}
\leq
\lambda_m^{(i)}(p;\Sigma)
\left(
\frac{|\Sigma|}{m |\Omega|}
\right)^{p/N}
\quad \text{for any}~ m \in \mathbb{N}.
\end{equation}
Moreover, if $\{\lambda_k^{(i)}(p;\Omega)\}$ also obeys a similar lower bound
\begin{equation}\label{eq:weyl1xx}
C_{\mathcal{W},\text{low}} |\Omega|^{-p/N} k^{p/N}
\leq
\lambda_k^{(i)}(p;\Omega) 
\quad \text{as}~ k \to +\infty,
\end{equation}
where $C_{\mathcal{W},\text{low}} = C_{\mathcal{W},\text{low}}(p,N)>0$, 
then the Weyl law holds for $\{\lambda_k^{(i)}(p;\Omega)\}$ by the same arguments as in  \cite{mazur}, see Section~\ref{sec:weyl} and Remark~\ref{rem:weyl}. 
As a consequence, Theorem~\ref{theorem:pleijel-general} remains valid under these two additional assumptions. 

In fact, if $m=1$, then \eqref{eq:prop:Clgen} holds without requiring the superadditivity assumption, since only compact symmetric subsets of $\mathcal{S}(\Omega)$ homeomorphic to $\mathbb{S}^{k-1}$ are involved in the arguments. 
In particular, Corollary~\ref{cor:m=1} is valid for any $i$ satisfying \ref{I1}, \ref{I2}, \ref{I4}. 

Recall from Remark~\ref{rem:weyl} that the Weyl law for the Krasnoselskii genus eigenvalues remains an open problem. However, noting that $\lambda_k^{G}(p;\Omega) \leq \lambda_k^{F}(p;\Omega)$ for any $k \in \mathbb{N}$ (see \cite[p.~195]{DR}), we get the upper bound
\begin{equation}\label{eq:prop:Clgen2}
\limsup_{k\to +\infty}
\frac{\lambda_k^{G}(p;\Omega)}{k^{p/N}}
\leq
\lambda_m^{F}(p;\Sigma)
\left(
\frac{|\Sigma|}{m |\Omega|}
\right)^{p/N}
\quad \text{for any}~ m \in \mathbb{N}.
\end{equation}
This implies that Propositions~\ref{prop:continuity} and \ref{prop:weyl2} remain valid also for the Krasnoselskii genus eigenfunctions. 
More generally, Propositions~\ref{prop:continuity} and \ref{prop:weyl2} hold for $\{\lambda_k^{(i)}(p;\Omega)\}$, provided $\lambda_k^{G}(p;\Omega) \leq \lambda_k^{(i)}(p;\Omega) \leq \lambda_k^{F}(p;\Omega)$ for any $k \in \mathbb{N}$. 
An example is the cohomological index, see \eqref{eq:lll0}.


\addcontentsline{toc}{section}{\refname}
\small
	
\begin{thebibliography}{11}

\bibitem{ananetsouli}
Anane, A., \& Tsouli, N. (1996). On the second eigenvalue of the $p$-Laplacian. In ``Nonlinear
Partial Differential Equations (From a conference in F\'es, Morocco, 1994)'' (A. Benkirane
and J.-P. Gossez, Ed.), Pitman Research Notes in Math., Vol. 343, pp. 1-9, Longman,
Harlow, 1996.

\bibitem{ABP}
Audoux, B., Bobkov, V., \& Parini, E. (2018). On multiplicity of eigenvalues and symmetry of eigenfunctions of the $p$-Laplacian. Topological Methods in Nonlinear Analysis, 51(2), 565-582.
\doi{10.12775/TMNA.2017.055}

\bibitem{APA} Garc\'ia Azorero, J. P., \& Peral Alonso, I. (1987). Existence and nonuniqueness for the $p$-Laplacian. Communications in Partial Differential Equations, 12(12), 1389-1430.
\doi{10.1080/03605308708820534}

\bibitem{BM}
	B\'erard, P., \& Meyer, D. (1982). In\'egalit\'es isop\'erim\'etriques et applications. Annales scientifiques de l'\'Ecole Normale Sup\'erieure, 15(3), 513-541.
	\url{https://eudml.org/doc/urn:eudml:doc:82104}

\bibitem{Bob2}
Bobkov, V. (2018). On exact Pleijel's constant for some domains. Documenta Mathematica, 23, 799-813.
\doi{10.4171/DM/634}

\bibitem{BobParRot}
Bobkov, V., \& Parini, E. (2021). On the Cheeger problem for rotationally invariant domains. Manuscripta Mathematica, 166(3), 503-522.
\doi{10.1007/s00229-020-01260-9}

\bibitem{coffman-classif}
Coffman, C. V. (1991). Classification of balanced sets and critical points of even functions on spheres. Transactions of the American Mathematical Society, 326(2), 727-747.
\doi{10.1090/S0002-9947-1991-1007802-0}


\bibitem{conwaysloane}
Conway, J. H., \& Sloane, N. J. A. (1999). Sphere packings, lattices and groups. 
Springer.
\doi{10.1007/978-1-4757-6568-7}


\bibitem{cuesta}
Cuesta, M. (2000). On the Fu\v{c}\'ik spectrum of the Laplacian and $p$-Laplacian.
In Proceedings of the ``2000 Seminar in Differential Equations'', Kvilda (Czech Republic), 2000.
\url{http://www-lmpa.univ-littoral.fr/~cuesta/articles/kavilda0.pdf}

\bibitem{CuestaFucik}
		Cuesta, M., De Figueiredo, D., \& Gossez, J. P. (1999). The beginning of the Fu\v{c}ik spectrum for the $p$-Laplacian. Journal of Differential Equations, 159(1), 212-238.
		\doi{10.1006/jdeq.1999.3645}

\bibitem{degiovannimarzocchi2}
Degiovanni, M., \& Marzocchi, M. (2014). Limit of minimax values under $\Gamma$-convergence. Electronic Journal of Differential Equations, 2014(266), 1-19. 
\url{https://ejde.math.txstate.edu/Volumes/2014/266/degiovanni.pdf}

\bibitem{degiovannimarzocchi}
Degiovanni, M., \& Marzocchi, M. (2015). 
On the dependence on $p$ of the variational eigenvalues of the $p$-Laplace operator. Potential Analysis, 43(4), 593-609.
\doi{10.1007/s11118-015-9487-0}


\bibitem{diazsaa}
			D\'iaz, J. I., \& Sa\'a, J. E. (1987). Existence et unicit\'e de solutions positives pour certaines \'equations elliptiques quasilin\'eaires. 
			Comptes Rendus de l'Acad\'emie des Sciences. S\'erie I, Math\'ematique, 305(12), 521-524.
			\url{https://gallica.bnf.fr/ark:/12148/bpt6k5494240w/f31.item}

\bibitem{DR0}
		Dr\'abek, P., \& Robinson, S. B. (1999). Resonance problems for the $p$-Laplacian. Journal of Functional Analysis, 169(1), 189-200.
		\doi{10.1006/jfan.1999.3501}

\bibitem{DR} Dr\'abek, P., \& Robinson, S. B. (2002). On the generalization of the Courant nodal domain theorem. Journal of Differential Equations, 181(1), 58-71.
\doi{10.1006/jdeq.2001.4070}

\bibitem{dugunji}
Dugundji, J. (1951). An extension of Tietze's theorem. Pacific Journal of Mathematics, 1(3), 353-367.
\doi{10.2140/pjm.1951.1.353}

\bibitem{fadell}
Fadell, E. R., \& Rabinowitz, P. H. (1977). Bifurcation for odd potential operators and an alternative topological index. Journal of Functional Analysis, 26(1), 48-67.
\doi{10.1016/0022-1236(77)90015-5}


\bibitem{F} Friedlander, L. (1989). Asymptotic behaviour of the eigenvalues of the $p$-Laplacian. Communications in Partial Differential Equations, 14(8-9), 1059-1069.
\doi{10.1080/03605308908820643}

\bibitem{Helffer}
	Helffer, B., \& Hoffmann-Ostenhof, T. (2015). A review on large $k$ minimal spectral $k$-partitions and Pleijel's Theorem. 
	Spectral theory and partial differential equations, 39-57, 
	Contemporary Mathematics, 640, AMS, 2015. 
	\doi{10.1090/conm/640/12841}

\bibitem{JLM} 
Juutinen, P., Lindqvist, P., \& Manfredi, J. J. (1999). The $\infty$-eigenvalue problem. Archive for Rational Mechanics and Analysis, 148(2), 89-105.
\doi{10.1007/s002050050157}

\bibitem{kawohl-lind}
		Kawohl, B., \& Lindqvist, P. (2006). Positive eigenfunctions for the $p$-Laplace operator revisited. Analysis -- International Mathematical Journal of Analysis and its Application, 26(4), 545-550.
		\doi{10.1524/anly.2006.26.4.545}

\bibitem{Leonardi2015}
Leonardi, G. P. (2015). An overview on the Cheeger problem. 
New Trends in Shape Optimization, International Series of Numerical Mathematics,
vol.~166, Birkh\"auser, 2015, 117--139.
\doi{10.1007/978-3-319-17563-8\_6}

\bibitem{mazur}
Mazurowski, L. (2025). A Weyl law for the $p$-Laplacian. Journal of Functional Analysis, 288(3), 110734.
\doi{10.1016/j.jfa.2024.110734}

\bibitem{Parini1}
Parini, E. (2010). The second eigenvalue of the $p$‐Laplacian as $p$ goes to $1$. International Journal of Differential Equations, 2010(1), 984671.
\doi{10.1155/2010/984671}

\bibitem{perera}
Perera, K. (2003). Nontrivial critical groups in $p$-Laplacian problems via the Yang index. 
Topological Methods in Nonlinear Analysis, 21(2), 301-310.
\url{https://apcz.umk.pl/TMNA/article/view/TMNA.2003.018}

\bibitem{perera-book}
Perera, K., Agarwal, R. P., \& O'Regan, D. (2010). Morse theoretic aspects of $p$-Laplacian type operators. American Mathematical Soc.
\doi{10.1090/surv/161}

\bibitem{pleijel} Pleijel, A. (1956). Remarks on Courant's nodal line theorem. Communications on Pure and Applied Mathematics, 9(3), 543-550.
\doi{10.1002/cpa.3160090324}

\bibitem{polya}
P\'olya, G. (1961). On the eigenvalues of vibrating membranes. 
Proceedings of the London Mathematical Society, 3(11), 419-433.
\doi{10.1112/plms/s3-11.1.419}


\bibitem{rabinowitz}
Rabinowitz, P. H. (1986). Minimax methods in critical point theory with applications to differential equations. American Mathematical Soc.
\doi{10.1090/cbms/065}

\bibitem{tolksdorf}
Tolksdorf, P. (1984). Regularity for a more general class of quasilinear elliptic equations. Journal of Differential equations, 51(1), 126-150. 
\doi{10.1016/0022-0396(84)90105-0}

\end{thebibliography}
\end{document}